\documentclass[12pt, a4paper]{amsart}
\usepackage{amsmath}
\usepackage{geometry,amsthm,graphics,tabularx,amssymb,
shapepar}
\usepackage{amscd}
\usepackage[usenames]{color}

\usepackage[all,2cell,dvips]{xy}

\newcommand{\BC}{{\mathbb {C}}}

\newcommand{\CF}{{\mathcal {F}}}

\newcommand{\CH}{{\mathcal {H}}}

\newcommand{\CJ}{{\mathcal {J}}}

\newcommand{\RD}{{\mathrm {D}}}

\newcommand{\RU}{{\mathrm {U}}}

\newcommand{\End}{{\mathrm{End}}}

\newcommand{\Hom}{{\mathrm{Hom}}}

\renewcommand{\Im}{{\mathrm{Im}}}

\newcommand{\Ker}{{\mathrm{Ker}}}

\newcommand{\rh}{{\mathrm{h}}}

\newcommand{\Sp}{{\mathrm{Sp}}}

\newcommand{\vsp}{{\vspace{0.2in}}}

\newcommand{\con}{\textit{C}}

\newcommand{\oB}{\operatorname{B}}

\newcommand{\oJ}{\operatorname{J}}
\newcommand{\oH}{\operatorname{H}}

\newcommand{\oU}{\operatorname{U}}

\newcommand{\oZ}{\operatorname{Z}}

\newcommand{\g}{\mathfrak g}
\newcommand{\gC}{{\mathfrak g}_{\C}}

\newcommand{\h}{\mathfrak h}
\newcommand{\p}{\mathfrak p}

\renewcommand{\l}{\mathfrak l}

\newcommand{\s}{\mathfrak s}

\newcommand{\z}{\mathfrak z}

\newcommand{\C}{\mathbb{C}}
\newcommand{\R}{\mathbb R}

\newcommand{\abs}[1]{\lvert#1\rvert}

\renewcommand{\rh}{\mathrm{h}}
\newcommand{\rz}{\mathrm{z}}
\newcommand{\rw}{\mathrm{w}}

\newcommand{\la}{\langle}
\newcommand{\ra}{\rangle}

\newcommand{\be}{\begin {equation}}
\newcommand{\ee}{\end {equation}}
\newcommand{\bee}{\begin {equation*}}
\newcommand{\eee}{\end {equation*}}

\renewcommand{\th}{\textrm{h}}

\theoremstyle{Theorem}
\newtheorem{introconjecture}{Conjecture}

\newtheorem{introtheorem}[introconjecture]{Theorem}
\newtheorem{introcorollary}[introconjecture]{Corollary}

\theoremstyle{Theorem}
\newtheorem{lem}{Lemma}[section]
\newtheorem{corl}[lem]{Corollary}

\newtheorem{leml}[lem]{Lemma}
\newtheorem{prpl}[lem]{Proposition}

\theoremstyle{Theorem}
\newtheorem{prp}{Proposition}[section]

\newtheorem{lemp}[prp]{Lemma}

\newtheorem{prpp}[prp]{Proposition}

\theoremstyle{Plain}

\theoremstyle{Definition}

\begin{document}

\title{On representations of real Jacobi groups}

\author [B. Sun] {Binyong Sun}
\address{Academy of Mathematics and Systems Science\\
Chinese Academy of Sciences\\
Beijing, 100190,  P.R. China} \email{sun@math.ac.cn}



\begin{abstract}

We consider a category of continuous Hilbert space representations
and a category of smooth Fr\'{e}chet representations, of a real
Jacobi group $G$. By Mackey's theory, they are respectively
equivalent to certain categories of representations of a real
reductive group $\widetilde L$.  Within these categories, we show
that the two functors of taking smooth vectors for $G$, and for
$\widetilde L$, are consistent with each other. By using
Casselman-Wallach's theory of smooth representations of real
reductive groups, we define matrix coefficients for distributional
vectors of certain representations of $G$. We also formulate
Gelfand-Kazhdan criteria for Jacobi groups which could be used to
prove the multiplicity one theorem for Fourier-Jacobi models.

\end{abstract}

\maketitle

\section{Introduction}

By a (complex) representation of a Lie group, we mean a continuous linear
action of it on a complete locally convex space. In this paper, a
locally convex space means a complex topological vector space which
is Hausdorff and locally convex. When no confusion is possible, we
do not distinguish a representation with its underlying space. Representations of reductive groups are studied intensively in the literature. These are the most interesting groups from the point of
view of the Langlands program. But there is another family of groups which show up quite often in number theory, namely, Jacobi groups.

We work in the setting of Nash groups. By a Nash group, we mean a group which is simultaneously a Nash manifold so that all group operations (the multiplication and the inversion) are Nash maps. The reader is referred to \cite{Sh, Sh2} for details on Nash manifolds and Nash groups. Let $G$ be a Nash group. A finite dimensional real representation $E$ of $G$ is said to a Nash representation if the action map $G\times E\rightarrow E$ is Nash. A Nash group is said to be almost linear if it admits a Nash representation with finite kernel. It is said to be unipotent if it admits a faithful Nash representation so that all group elements act as unipotent operators. Every unipotent Nash group is connected and simply connected. As in the case of linear algebraic groups, every almost linear Nash group has a unipotent radical, namely, a largest unipotent normal Nash subgroup of it. (Recall that every Nash subgroup is automatically closed.) An almost linear Nash group is said to be reductive if its unipotent radical is trivial. Assume that $G$ is almost linear, and denote by $H$ its unipotent radical. The center of $G$ is also an almost linear Nash group. Its unipotent radical $Z$ is called the unipotent center of $G$. This equals to the intersection of $H$ with the center of $G$. If $\widetilde G\rightarrow G$ is a finite cover of Lie groups. Then $\widetilde G$ is uniquely a Nash group so that the covering map is Nash. In this case, the unipotent radical (and the unipotent center) of $G$ and $\widetilde G$ are canonically identified.

As usual, we use the corresponding lower case German letter to denote the Lie algebra of a Lie group. We say that the almost linear Nash group $G$ is a real Jacobi group if $[\h,\h]\subset \z$ and the map
\[
  [\,,\,]: \h/\z\times \h/\z\rightarrow \z
\]
is non-degenerate in the sense that the alternating form
\[
  \phi\circ [\,,\,]: \h/\z\times \h/\z\rightarrow \R
\]
is non-degenerate for some linear functional $\phi:\z\rightarrow \R$. By definition, every reductive Nash group is a real Jacobi group. It is easy to see that all real Jacobi groups are unimodular.

Now assume that $G$ is a real Jacobi group. Fix a unitary character $\psi$ on $Z$ which is generic in the sense that the alternating form
\begin{equation}\label{sym}
          \frac{d\psi}{\sqrt{-1}} \circ [\,,\,]: \h/\z\times \h/\z \rightarrow \R
\end{equation}
is non-degenerate, where $d\psi$ denotes the differential of $\psi$. A representation of $G$ is said to be a $\psi$-representation if $Z$ acts through the character $\psi$.

Among all representations, unitary ones are most important for many
applications. But non-unitary representations also occur naturally
even in the study of unitary ones. In general, we will consider $\psi$-representations of $G$ on Hilbert spaces so that $H$ acts by unitary operators. Denote by
$\CH\mathbf{mod}_{G,\psi}$ the category of these representations.
Morphisms in this category are $G$-intertwining continuous linear
maps (which may or may not preserve the inner products).

Smooth representations are also important for many purposes. In this
paper, we are very much concerned with smooth Fr\'{e}chet
$\psi$-representations of $G$ of moderate growth. (A representation is said to be Fr\'{e}chet if
its underlying space is. See Section \ref{smoothing} for the
notion of smooth representations, and see Section \ref{functions}
for the moderate growth condition.) Following F. du Cloux
(\cite{du91}), denote by $\CJ\mathbf{mod}_{G,\psi}$ the category of
these representations.

Given any representation $V$ in $\CH\mathbf{mod}_{G,\psi}$, it turns
out (cf. Corollary \ref{modg}) that its smooth vectors $V_{\infty}$
form a representation in $\CJ\mathbf{mod}_{G,\psi}$ (see Section
\ref{smoothing} for the notion of smooth vectors). Therefore we get
a functor (the smoothing functor)
\begin{equation}\label{sm1}
  (\,\cdot\,)_{\infty}:\CH\mathbf{mod}_{G,\psi}\rightarrow
  \CJ\mathbf{mod}_{G,\psi}.
\end{equation}
We are aimed to show that the smoothing functor \eqref{sm1} can be
identified with the smoothing functor for a reductive group.

To be more precise, use the alternating form \eqref{sym}, we form
the real symplectic group $\Sp(\h/\z)$ and its metaplectic double cover
$\widetilde{\Sp}(\h/\z)$. The adjoint action induces a homomorphism
\[
 L:=G/H \rightarrow \Sp(\h/\z).
\]
Define the fibre product
\[
  \widetilde L:=L\times_{\Sp(\h/\z)}\widetilde{\Sp}(\h/\z).
\]
This is a double cover of $L$. All the groups $\Sp(\h/\z)$, $\widetilde{\Sp}(\h/\z)$, $L$ and $\widetilde L$ are canonically reductive Nash groups.

A representation of $\widetilde L$ is said to be
genuine if the nontrivial element in the kernel of the covering map $\widetilde L\rightarrow L$ acts as the scalar multiplication by $-1$. This terminology applies to other double covers of groups.
Denote by $\CH\mathbf{mod}_{\widetilde{L},\textrm{gen}}$ the
category of genuine representations of $\widetilde{L}$ on Hilbert
spaces, and by $\CJ\mathbf{mod}_{\widetilde{L},\textrm{gen}}$ the
category of genuine smooth Fr\'{e}chet representations of
$\widetilde{L}$ of moderate growth. As in the case of Jacobi groups,
we still have the smoothing functor:
 \begin{equation}\label{sm2}
  (\,\cdot\,)_{\infty}:\CH\mathbf{mod}_{\widetilde{L},\textrm{gen}}\rightarrow
  \CJ\mathbf{mod}_{\widetilde{L},\textrm{gen}}.
\end{equation}

Put
\[
  \widetilde G:=G\times_L\widetilde{L}.
\]
This is a double cover of $G$ so that $\widetilde G/H=\widetilde L$. In order to relate $\psi$-representations of $G$ to genuine representations of $\widetilde{L}$, we fix a unitary oscillator representation $\omega_{\psi}^\th$ of $\widetilde G$ corresponding to $\psi$: $\omega_\psi^\th$ is a genuine unitary $\psi$-representation of $\widetilde G$ which remains irreducible when restricted to $H$. Such a representation always exists (see Section \ref{jacobi}), and all others are unitarily isomorphic to twists of $\omega_{\psi}^\th$ by unitary characters of $L$. Denote by $\omega_{\psi}$ the smoothing of $\omega_{\psi}^\th$. We construct in Section \ref{jacobi} four functors:
\begin{eqnarray}
 \label{fu1}\Hom_H (\omega_{\psi}^{\th},\cdot) &:& \CH\mathbf{mod}_{G,\psi}\rightarrow
\CH\mathbf{mod}_{\widetilde{L},\textrm{gen}}, \\
    \label{fu2}      \omega_{\psi}^\th\widehat{\otimes}_\th\,\cdot           &:&   \CH\mathbf{mod}_{\widetilde{L},\textrm{gen}}\rightarrow
\CH\mathbf{mod}_{G,\psi},  \\
   \label{fu3} \Hom_H (\omega_{\psi},\cdot) &:& \CJ\mathbf{mod}_{G,\psi}\rightarrow
\CJ\mathbf{mod}_{\widetilde{L},\textrm{gen}}, \\
    \label{fu4}    \omega_{\psi}\widehat{\otimes}\,\cdot           &:&   \CJ\mathbf{mod}_{\widetilde{L},\textrm{gen}}\rightarrow
\CJ\mathbf{mod}_{G,\psi}.
                      \end{eqnarray}
The usual meanings of the topological tensor products ``$\widehat{\otimes}_\th$" and ``$\widehat{\otimes}$" will be explained in Section \ref{hilt}.

One purpose of this paper is to clarify the following

\begin{introtheorem}\label{thm0}
The functors \eqref{fu1} and \eqref{fu2} are inverse to each other,
the functors \eqref{fu3} and \eqref{fu4} are inverse to each other,
and the diagrams
\[
  \begin{CD}
            \CH\mathbf{mod}_{G,\psi}@>(\cdot)_{\infty}>> \CJ\mathbf{mod}_{G,\psi}\\
            @V\Hom_H (\omega_{\psi}^{\th},\cdot) VV           @VV\Hom_H (\omega_{\psi},\cdot)V\\
           \CH\mathbf{mod}_{\widetilde{L},\textrm{gen}}@>(\cdot)_{\infty}>> \CJ\mathbf{mod}_{\widetilde{L},\textrm{gen}}
  \end{CD}
\]
and
\[
  \begin{CD}
            \CH\mathbf{mod}_{G,\psi}@>(\cdot)_{\infty}>> \CJ\mathbf{mod}_{G,\psi}\\
            @A\omega_{\psi}^\th\widehat{\otimes}_\th AA           @AA\omega_{\psi}\widehat{\otimes} A\\
           \CH\mathbf{mod}_{\widetilde{L},\textrm{gen}}@>(\cdot)_{\infty}>> \CJ\mathbf{mod}_{\widetilde{L},\textrm{gen}}
  \end{CD}
\] commute.
\end{introtheorem}

More precisely, Theorem \ref{thm0} says that we have the following
natural identifications of representations:
\begin{eqnarray*}
                                           \omega_{\psi}^\th\widehat{\otimes}_\th \, \Hom_H(\omega_{\psi}^\th, V^\th) &=& V^\th, \\
                                           \Hom_H(\omega_{\psi}^\th,\omega_{\psi}^\th\widehat{\otimes}_\th E^\th) &=&   E^\th, \\
                                               \omega_{\psi}\widehat{\otimes} \, \Hom_H(\omega_{\psi}, V) &=& V, \\
                                           \Hom_H(\omega_{\psi},\omega_{\psi}\widehat{\otimes} E) &=&   E, \\
                                             \Hom_H(\omega_{\psi}^\th, V^\th)_\infty &=& \Hom_H(\omega_{\psi}, V^\th_\infty) , \\
                                           (\omega_{\psi}^\th\widehat{\otimes}_\th E^\th)_\infty &=& \omega_{\psi}\widehat{\otimes}
                                           E^\th_\infty,
                                            \end{eqnarray*}
for all representations $V^\th$ in $\CH\mathbf{mod}_{G,\psi}$,
$E^\th$ in $\CH\mathbf{mod}_{\widetilde{L},\textrm{gen}}$,  $V$ in
$\CJ\mathbf{mod}_{G,\psi}$, and $E$ in
$\CJ\mathbf{mod}_{\widetilde{L},\textrm{gen}}$. The first two
assertions of Theorem \ref{thm0} are in some sense well known
(Mackey's theory, cf. \cite[Proposition 4.3.9]{du91}). They are
rather direct consequences of the work of F. du Cloux (\cite{du88,
du89, du91}). The last assertion of Theorem \ref{thm0} generalizes
the expectation of R. Berndt in \cite[Page 185]{Be94}.

\vsp

We say that a representation of a reductive Nash group is a Casselman-Wallach representation if it is Fr\'{e}chet, smooth, of moderate growth, admissible and $\oZ$-finite. Here $\oZ$ is the center of the universal enveloping algebra of the complexified Lie algebra of the group. The reader may consult \cite{Cas}, \cite[Chapter 11]{Wa} and \cite{BK} for more details about Casselman-Wallach representations.

Denote by $\CF\CH_{\widetilde{L},\textrm{gen}}$ the category of
genuine Casselman-Wallach representations of $\widetilde{L}$. This is a full subcategory of $\CJ\mathbf{mod}_{\widetilde{L},\textrm{gen}}$. Denote by $\CF\CH_{G,\psi}$ the full subcategory of $\CJ\mathbf{mod}_{G,\psi}$ corresponding to  $\CF\CH_{\widetilde{L},\textrm{gen}}$, under the functors \eqref{fu3} and \eqref{fu4}. Objects of this category are called Casselman-Wallach $\psi$-representations of $G$. These representations should be useful in the study of Jacobi forms (cf. \cite{EZ}). It is important to note that the underlying spaces of all Casselman-Wallach $\psi$-representations are nuclear.

As an application of Theorem \ref{thm0}, we define matrix coefficients  for distributional vectors in Casselman-Wallach $\psi$-representations, as what follows. Denote by $\con^{\,\xi}(G)$ the space of tempered smooth functions
on $G$, and by $\con^{-\xi}(G)$ the locally convex space of tempered
generalized functions on $G$. The later contains the former as a
dense subspace (see Section \ref{functions} for precise definitions
of these spaces).

Let $U$ be a representation in $\CF\CH_{G,\psi}$ and let $V$ be a
representation in $\CF\CH_{G,\bar{\psi}}$ which are contragredient
to each other, namely, a non-degenerate $G$-invariant continuous
bilinear form
\[
  \la \,,\,\ra: U\times V\rightarrow \C
\]
is given. Note that every representation in $\CF\CH_{G,\psi}$ has a
unique contragredient representation in $\CF\CH_{G,\bar{\psi}}$
(Proposition \ref{dual}). Denote by $U^{-\infty}$ the strong dual of
$V$. It is a smooth representation of $G$ containing $U$ as a dense
subspace. Similarly, denote by $V^{-\infty}$ the strong dual of $U$.

For any $u\in U$, $v\in V$, the (usual) matrix coefficient
$c_{u\otimes v}$ is defined by
\begin{equation}\label{matrixo}
  c_{u\otimes v}(g):=\la gu,v\ra, \quad g\in G.
\end{equation}
It is a function in $\con^{\,\xi}(G)$. The following result is proved in \cite[Theorem 2.1]{SZ} for real reductive groups.

\begin{introcorollary}\label{thm1} With the notation as above, the
matrix coefficient map
\begin{equation}\label{umatrix}
 \begin{array}{rcl}
 U\times V&\rightarrow& \con^{\,\xi}(G),\\
           (u,v)&\mapsto & c_{u\otimes v}
 \end{array}
\end{equation}
extends to a continuous bilinear map
\[
  U^{-\infty}\times V^{-\infty}\rightarrow \con^{-\xi}(G),
\]
and the induced $G\times G$ intertwining linear map
\begin{equation}\label{mapc}
   c \,:\, U^{-\infty}\widehat \otimes V^{-\infty}\rightarrow \con^{-\xi}(G)
\end{equation}
is a topological homomorphism with closed image.
\end{introcorollary}
Recall that a linear map $\phi: E\rightarrow F$ of locally convex spaces is called a topological homomorphism if the induced
linear isomorphism $E/\Ker(\phi)\rightarrow \Im (\phi)$ is a
homeomorphism, where $E/\Ker(\phi)$ is equipped with the quotient
topology of $E$, and the image $\Im(\phi)$ is equipped with the
subspace topology of $F$. The action of $G\times G$ on
$\con^{-\xi}(G)$ is given by
\[
  ((g_1,g_2).f)(x):=f(g_2^{-1}xg_1).
\]

In particular, Corollary \ref{thm1} defines characters of Casselman-Wallach $\psi$-representations (as tempered generalized functions on $G$). It also implies that irreducible Casselman-Wallach $\psi$-representations are determined by their characters (\cite[Remark 2.2]{SZ}).

\vsp The following form of Gelfand-Kazhdan criteria is a rather direct consequence of
Corollary \ref{thm1}. See \cite[Theorem 2.3]{SZ} for a proof.

\begin{introcorollary}\label{gelfand} Let $S_1$ and $S_2$ be two closed subgroups of the real Jacobi group $G$, with continuous (non-necessarily unitary) characters
\[
  \chi_i: S_i\rightarrow \BC^\times,\quad i=1,2.
\]
Assume that there is a Nash anti-automorphism $\sigma$ of $G$
such that for every $f\in \con^{-\xi}(G)$ which is an eigenvector of
$\operatorname{U}(\gC)^G$, the conditions
\[
\left\{
       \begin{array}{l}
           f(zx)=\psi(z)f(x),\quad z\in Z, \medskip\\
            f(sx)=\chi_{1}(s)f(x), \quad s\in S_1, \textrm{ and}\medskip\\
           f(xs)=\chi_{2}(s)^{-1}f(x), \quad s\in S_2
       \end{array}
\right.
\]
imply that
\[
   f(x^\sigma)=f(x).
\]
Then for every pair of irreducible representations $U$ in
$\CF\CH_{G,\psi}$ and $V$ in $\CF\CH_{G,\bar \psi}$ which are
contragredient to each other, one has that
\begin{equation*}\label{dhom}
  \dim \Hom_{S_1}(U, \chi_{1}) \, \dim \Hom_{S_2}
  (V,\chi_{2})\leq 1.
\end{equation*}

\end{introcorollary}

Here and henceforth, a subscript ``$\C$" indicates the complexification of a real Lie algebra, the universal enveloping algebra  $\operatorname{U}(\gC)$  is identified with the algebra of left invariant differential operators on $G$, and
$\operatorname{U}(\gC)^G$  is identified with the algebra of
bi-invariant differential operators on $G$.

\vsp

Another purpose of this paper is to have the following criterion for a strong Gelfand pair, which is used in \cite{SZ2} to prove the multiplicity one theorem for Fourier-Jacobi models. As in the proof of \cite[Corollary 2.5]{SZ}, the criterion is implied by Corollary \ref{gelfand}. We shall not go to the details. 

\begin{introcorollary}\label{strongg}
Let $G'$ be a Nash subgroup of $G$ which is also a real Jacobi grouop. Fix a generic unitary character $\psi'$ of the unipotent center $Z'$ of $G'$. Assume that there exists a Nash anti-automorphism $\sigma$ of $G$ preserving $G'$ with the following property: every tempered generalized function on $G$ which is invariant under the adjoint action of $G'$ is automatically $\sigma$-invariant. Then for every irreducible Casselman-Wallach $\psi$-representation $U$ of $G$, and every irreducible Casselman-Wallach $\psi'$-representation $U'$ of $G'$, the space of $G'$-invariant continuous bilinear functionals on $U\times U'$ is at most one dimensional.

\end{introcorollary}

\vsp

\noindent Acknowledgements: This paper is an outcome of discussions
between Avraham Aizenbud and the author. In particular, the author
learned Lemma \ref{tensor0} and Lemma \ref{subalge} from him. The
author is very grateful to him. The work was partially supported by
NSFC grants 10801126 and 10931006.

\section{Preliminaries}\label{pre}

In this section, we fix some notations and terminologies and recall
some general results which will be used later in this paper.

\subsection{Topological tensor products}\label{hilt}
Let $E$ and $F$ be two locally convex spaces. There are at least
four useful locally convex topologies one can put on the algebraic
tensor product $E\otimes F$, namely, the inductive tensor product
$E\otimes_{\textrm{i}}F$, the projective tensor product
$E\otimes_{\pi}F$, the epsilon tensor product $E\otimes_{\epsilon}F$
and the Hilbert tensor product $E\otimes_{\textrm{h}}F$. As usual,
their completions are denoted by $E\widehat \otimes_{\textrm{i}}F$,
$E\widehat \otimes_{\pi}F$, $E\widehat \otimes_{\epsilon}F$ and
$E\widehat \otimes_{\textrm{h}}F$, respectively. The inductive
tensor product plays no role in this paper. The projective tensor
product and the epsilon tensor product are more commonly used and
the reader is referred to \cite{Th01} for a concise treatment. A
fundamental theorem of Grothedieck says that they coincide when
either $E$ or $F$ is nuclear. If this is the case, we simply write
\[
  E\widehat \otimes F:=E\widehat \otimes_{\pi}F=E\widehat
  \otimes_{\epsilon}F.
\]

We are more concerned with the Hilbert tensor product. Its topology
is defined by the family
\[
  \{\,\la \,,\,\ra_\mu\otimes \la \,,\,\ra_\nu\,\}
\]
of non-negative Hermitian forms on $E\otimes F$, where $\la
\,,\,\ra_\mu$ and $\la \,,\,\ra_\nu$ runs through all non-negative
continuous Hermitian forms on $E$ and $F$, respectively.

Recall that a locally convex space is said to be Hilbertizable if
its topology is defined by a family of non-negative Hermitian forms
on it. The following fact is well known.

\begin{lem}\label{tensor0}
If $E$ is nuclear and $F$ is Hilbertizable, then
\[
  E\widehat \otimes_{\emph{h}}F=E\widehat \otimes F
\]
as locally convex spaces.
\end{lem}
\begin{proof}
We indicate the idea of the proof for the convenience of the reader.
In general, the Hilbert topology is coarser than the projective
topology. If both $E$ and $F$ are Hilbertizable, then the epsilon
topology is coarser than the Hilbert topology. The lemma then
follows by Grothedieck's theorem and the well know fact that every
nuclear locally convex space is Hilbertizable (cf. \cite[Corallary
3.19]{Th01}).
\end{proof}

\subsection{Smoothing representations}\label{smoothing}
In this subsection, $G$ is an arbitrary Lie group. Let $V$ be a
representation of $G$ (recall from the Introduction that every
representation space is assumed to be complete). It is said to be
smooth if the action map $G\times V\rightarrow V$ is smooth as a map
of infinite dimensional manifolds. The notion of smooth maps in
infinite dimensional setting may be found in \cite{GN}, for example.
In general, denote by $\con(G;V)$ the space of continuous functions
on $G$ with values in $V$. It is a complete locally convex space
under the usual topology of uniform convergence on compact sets, and
is a representation of $G$ under right translations:
\[
  (g.f)(x):=f(xg),\quad g,x\in G, \,f\in\con(G;V).
\]
Similarly, smooth $V$-valued functions form a complete locally
convex space $\con^{\,\infty}(G;V)$ under the usual smooth topology,
and is a smooth representation of $G$ under right translations.

Write
\[
  \gamma_v(g):=g.v, \quad v\in V,\,g\in G.
\]
The smoothing $V_\infty$ of $V$ is defined to be the representation
which fits to a cartesian diagram
\[
  \begin{CD}
           V_{\infty}@>>> \con^{\,\infty}(G;V)\\
            @VVV           @VVV\\
           V@>v\mapsto \gamma_v>>\con(G;V)
  \end{CD}
\]
in the category of representations of $G$. This is a smooth
representation of $G$. As a vector space, it consists of all $v\in
V$ such that $\gamma_v\in \con^{\,\infty}(G;V)$. The topology of
$V_\infty$ coincides with the subspace topology of
$\con^{\,\infty}(G;V)$. The universal enveloping algebra
$\oU(\g_\C)$ acts on $V_\infty$ as continuous linear operators by
\[
  X.v:=\textrm{the value of $X(\gamma_v)$ at the identity element $1\in
  G$}.
\]
The topology of $V_\infty$ is determined by the family
\[
  \{\abs{\,\cdot\,}_{X,\lambda}\}_{X\in \oU(\g_\C), \,\lambda\in
  \Lambda}
\]
of seminorms, where $\Lambda$ is the set of all continuous seminorms
on $V$, and
\[
  \abs{v}_{X,\lambda}:=\abs{X.v}_\lambda.
\]

The smoothing is clearly a functor from the category of
representations of $G$ to the category of smooth representations of
$G$. When $V$ itself is smooth, we have that $V=V_\infty$ as
representations of $G$. In general, $V_\infty$ is a dense subspace
of $V$.

\begin{lem}\label{densesmooth}
Let $V_0$ be a $G$-stable subspace of $V_\infty$. If it is dense in
$V$, then it is also dense in $V_\infty$.
\end{lem}
\begin{proof}Fix a left invariant Haar measure $dg$ on $G$. Let
$f\in \con^\infty_0(G)$ (a smooth function with compact support).
For all $v\in V$, write
\[
  f.v:=\int_G f(g)g.v\,dg.
\]
It is well know and also easy to see that $f.v\in V_\infty$ and the
linear map
\[
   V\rightarrow V_\infty,\quad v\mapsto f.v
\]
is continuous. Then the denseness of $V_0$ in $V$ implies the
denseness of $f.V_0$ in $f.V$, under the topology of $V_\infty$.
Consequently, under the topology of $V_\infty$,
\begin{equation}\label{densee}
  \sum_{f\in \con^\infty_0(G)} f.V_0\quad\textrm{is dense in}\quad \sum_{f\in \con^\infty_0(G)}
  f.V.
\end{equation}
Then the lemma follows by noting that the closure of $V_0$ (within
$V_\infty$) contains the first space of \eqref{densee}, and that the
second space of \eqref{densee} (the Garding subspace) is dense in
$V_\infty$.

\end{proof}

For every closed subgroup $S$ of $G$, write $V|_S:=V$, viewed as a
representation of $S$.
\begin{lem}\label{subsmooth}
Let $S$ and $S'$ be two closed subgroups of $G$. If the actions of
$\oU(\s_\C)$ and $\oU(\s'_\C)$ produce the same subalgebra of
$\End(V_\infty)$, then
\[
  (V|_S)_\infty=(V|_{S'})_\infty
\]
as locally convex spaces.
\end{lem}
\begin{proof}
Apply Lemma \ref{densesmooth} to the representation $V|_S$, we see
that $V_\infty$ is dense in $(V|_S)_\infty$. Therefore
$(V|_S)_\infty$ is the completion of $V_\infty$ under the seminorms
\[
  \{\abs{\,\cdot\,}_{X,\lambda}\}_{X\in \oU(\s_\C),\,\lambda\in
  \Lambda}.
\]
Similarly, $(V|_{S'})_\infty$ is the completion of $V_\infty$ under
the seminorms
\[
  \{\abs{\,\cdot\,}_{X,\lambda}\}_{X\in \oU(\s'_\C),\,\lambda\in
  \Lambda}.
\]
The assumption of the lemma implies that these two families of
seminorms are the same. Therefore the lemma follows.

\end{proof}

We say that a representation is Hibertizable if its underlying
space is. The smoothing of a  Hibertizable representation is also
Hibertizable. Let $V'$ be another representation of another Lie
group $G'$. Then $V\widehat{\otimes}_\pi V'$ is a representation of
$G\times G'$. If both $V$ and $V'$ are Hibertizable, then
$V\widehat{\otimes}_\textrm{h} V'$ is also a (Hilbertizable)
representation of $G\times G'$.

\begin{lem}\label{smooth3}
If both $V$ and $V'$ are smooth, then so is $V\widehat{\otimes}_\pi
V'$. If both $V$ and $V'$ are smooth and Hibertizable, then so is
$V\widehat{\otimes}_\rh V'$.
\end{lem}
\begin{proof}
We prove the firs assertion. The second one is proved similarly. The
algebraic tensor product $V\otimes V'$ is clearly contained in
$(V\widehat{\otimes}_\pi V')_\infty$. It is dense in
$V\widehat{\otimes}_\pi V'$ and is therefore also dense in
$(V\widehat{\otimes}_\pi V')_\infty$, by Lemma \ref{densesmooth}.
Note that $\oU(\g_\C\times \g_\C')$ acts continuously on
$V\otimes_\pi V'$. This implies that the topology of
$(V\widehat{\otimes}_\pi V')_\infty$ and of $V\widehat{\otimes}_\pi
V'$ have the same restriction to $V\otimes V'$. Consequently,
$(V\widehat{\otimes}_\pi V')_\infty$ and $V\widehat{\otimes}_\pi V'$
are both completions of $V\otimes V'$ with respect to a common
locally convex topology. Therefore they are the same and the first
assertion of the lemma follows.
\end{proof}

\begin{lem}\label{hs}
If both $V$ and $V'$ are Hibertizable, then
\begin{equation}\label{tens}
  (V\widehat{\otimes}_\rh V')_\infty=V_\infty\widehat{\otimes}_\rh
  V'_\infty.
\end{equation}
\end{lem}
\begin{proof}
The proof is similar to that of Lemma \ref{smooth3} and we will not
go to the details. The key point is that both sides of \eqref{tens}
contain $V_\infty\otimes V'_\infty$ as a dense subspace.
\end{proof}

It is not clear to the author whether the analog of \eqref{tens}
holds for projective tensor products.

\subsection{Representations of moderate growth}\label{functions}

We will use the notation of \cite{SZ} for function spaces, as what follows. Let $G$ be an almost linear Nash group. It is not assumed to be a real Jacobi group in this section. We say that a (complex valued) function $f$ on $G$ has moderate growth if its  absolute value is bounded by a positive Nash
function $\phi$ on $G$:
\[
   \abs{f(x)}\leq\phi(x) \quad \textrm{for all } x\in G.
\]
A smooth function $f\in \con^\infty(G)$ is said to be tempered if
$Xf$ has moderate growth for all $X\in\RU(\g_\BC)$. Denote by
$\con^{\,\xi}(G)$ the space of all tempered smooth functions on $G$.

A smooth function $f\in \con^\infty(G)$ is called Schwartz if
\[
  \abs{f}_{X,\phi}:=\mathrm{sup}_{x\in G} \,\phi(x)\,\abs{(Xf)(x)}<
  \infty
\]
for all $X\in \RU(\g_\BC)$, and all positive functions $\phi$ on $G$
of moderate growth. Denote by $\con^{\,\varsigma}(G)$ the space of
Schwartz functions on $G$. It is a nuclear Fr\'{e}chet space under
the seminorms $\{\abs{\,\cdot}_{X,\phi}\}$.  We define the nuclear
Fr\'{e}chet space $\RD^{\varsigma}(G)$ of Schwartz densities on $G$
similarly. Fix a Haar measure $dg$ on $G$, then the map
\[
  \begin{array}{rcl}
               \con^{\,\varsigma}(G)&\rightarrow& \RD^{\varsigma}(G),\\
                                    f&\mapsto &f\,dg
  \end{array}
\]
is a topological linear isomorphism. We define a tempered
generalized function on $G$ to be a continuous linear functional on
$\RD^{\varsigma}(G)$. Denote by $\con^{-\xi}(G)$ the space of all
tempered generalized functions on $G$, equipped with the strong dual
topology. This topology coincides with the topology of uniform
convergence on compact subsets of $\RD^{\varsigma}(G)$, due to the
fact that every bounded subset of a nuclear locally convex space is
relatively compact.  Note that $\con^{\,\xi}(G)$ is canonically
identified with a dense subspace of $\con^{-\xi}(G)$:
\[\con^{\,\xi}(G)\hookrightarrow \con^{-\xi}(G).\]

 A representation $V$ of $G$ is said to be of moderate growth if for
every continuous seminorm $\abs{\,\cdot\,}_\mu$ on $V$, there is a
positive function $\phi$ on $G$ of moderate growth and a
continuous seminorm $\abs{\,\cdot\,}_\nu$ on $V$ such that
\[
   \abs{gv}_\mu\leq\phi(g)\, \abs{v}_\nu,\quad \textrm{for all } g\in
   G, \  v\in V.
\]
It is easy to check that the smoothing of a representation of
moderate growth is still of moderate growth. For every
representation $V$ of $G$  of moderate growth, the bilinear map
\begin{equation}\label{acts}
  \begin{array}{rcl}
        \RD^{\varsigma}(G)\times V&\rightarrow&  V_\infty,\\
       \lambda=f(g)\,dg,\, v&\mapsto & \lambda.v:=\int_G f(g) g.v\,  dg
  \end{array}
\end{equation}
is well defined and continuous. The space $\RD^{\varsigma}(G)$ is an
associative algebra under the convolution operator $*$. Under the
map \eqref{acts}, both $V$ and $V_\infty$ are
$\RD^{\varsigma}(G)$-modules.

\subsection{Matrix coefficients for distributional
vectors}\label{smatrix} Let $G$ be an almost linear Nash group as in the last subsection. Let
$U^\th$ and $V^\th$ be two Hilbert space representations of $G$ of moderate growth which
are strongly dual to each other in the following sense: there is
given a $G$-invariant continuous bilinear map
\[
  \la\, ,\,\ra: U^\th\times V^\th\rightarrow \C
\]
which induces a topological isomorphism from $U^\th$ to the strong dual
of $V^\th$ (and hence a topological isomorphism from $V^\th$ to the strong dual of
$U^\th$). Denote by $U^{-\infty}$ the strong dual of the smoothing
$V^\th_\infty$ of $V^\th$. This is a locally convex space carrying a linear
action of $G$. There are canonical $G$-intertwining continuous
injections
\[
   U^\th_\infty\rightarrow U^\th\rightarrow U^{-\infty}.
\]
It is not clear to the author whether $U^\th$ is dense in $U^{-\infty}$
in general. Similarly, denote by $V^{-\infty}$ the strong dual of
the smoothing $U^\th_\infty$.

Define a bilinear map
\begin{equation}\label{act0}
  \begin{array}{rcl}
        \RD^{\varsigma}(G)\times U^{-\infty}&\rightarrow&
        U^\th=\Hom_\C(V^\th,\C),\\
       \lambda=f(g)\,dg,\, \phi&\mapsto & \lambda.\phi:=\left(v\mapsto \phi\left(\int_G f(g) g^{-1}.v\,
        dg\right)\right).
  \end{array}
\end{equation}
The following lemma is elementary and may be proved by the argument
of \cite[Section 3]{SZ}.

\begin{lem}\label{actl}
The map \eqref{act0} is well defined and its image is contained in
$U^\th_\infty$. The resulting bilinear map
\[
    \RD^{\varsigma}(G)\times U^{-\infty}\rightarrow U_\infty,\quad (\lambda, \phi)\mapsto
    \lambda.\phi
\]
is separately continuous and extends the continuous bilinear map
\eqref{acts} for the representation $U^{\th}$.
\end{lem}

Finally, we define the matrix coefficient map
\[
  \begin{array}{rcl}
      c:  U^{-\infty}\times V^{-\infty}&\rightarrow&
        \con^{-\xi}(G)=\Hom_\C(\RD^{\varsigma}(G),\C),\\
   \phi, \phi'&\mapsto & c_{\phi\otimes \phi'}:=\left(\lambda\mapsto \phi'(\lambda.\phi)\right).
  \end{array}
\]
This extends the ordinary
matrix coefficient map (see \eqref{matrixo}), and is separately
continuous (see the proof of \cite[Lemma 3.6]{SZ}).

\section{Representations of Heisenberg groups}

Let $\mathrm w$ be a finite dimensional real vector space with a non-degenerate skew-symmetric bilinear map
\[
  \la\,,\,\ra_{\mathrm w}: \mathrm w\times \mathrm w\rightarrow \mathrm z,\quad \textrm{where $\dim_\R(\mathrm z)=1$.}
\]
Let $\oH(\mathrm{w})=\mathrm{w}\times \mathrm{z}$ be the associated Heisenberg group, with group multiplication
\[
  (u,t)(u',t'):=(u+u', t+t'+\la u,u'\ra_{\mathrm w}).
\]
Let $\psi_\mathrm{z}$ be a non-trivial unitary character on $\mathrm{z}$. By Stone-Von
Neumann Theorem, up to isomorphism, there is a unique irreducible
unitary $\psi_{\mathrm z}$-representation $\omega_{\psi_{\mathrm z}}^\rh$ of
$\oH(\mathrm{w})$. Write $\omega_{\psi_{\mathrm z}}$ for its smoothing. It is well known that
this is a nuclear Fr\'{e}chet space.

\subsection{Smooth representations}
Recall that $\CJ\mathbf{mod}_{\oH(\mathrm{w}),\psi_\mathrm{z}}$ is
the category of smooth Fr\'{e}chet $\psi_{\mathrm z}$-representations of $\oH(\rw)$ of
moderate growth. The following
fact is fundamental to this paper.
\begin{prp}\label{factor1}
For every representation $V$ in
$\CJ\mathbf{mod}_{\oH(\rw),\psi_\rz}$,
$\Hom_{\oH(\rw)}(\omega_{\psi_\rz}, V)$ is a Fr\'{e}chet space under
the topology of uniform convergence on bounded sets. For every
Fr\'{e}chet space $E$, $\omega_{\psi_\mathrm{z}}\widehat{\otimes}E$
is a representation in $\CJ\mathbf{mod}_{\oH(\rw),\psi_\rz}$.
Further more, one has that
\[
  \omega_{\psi_\mathrm{z}}\widehat{\otimes} \, \Hom_{\oH(\mathrm{w})}(\omega_{\psi_\mathrm{z}},
  V)=V,
\]
and
\[
   \Hom_{\oH(\mathrm{w})}(\omega_{\psi_\mathrm{z}},\omega_{\psi_\mathrm{z}}\widehat{\otimes}
   E)=E.
\]
Consequently, the functors
\[
\Hom_{\oH(W)}(\omega_{\psi_\mathrm{z}},\,\cdot)\,
\quad\textrm{and}\quad \omega_{\psi_\mathrm{z}}\widehat{\otimes}
\,\cdot
 \]
are mutually inverse equivalences of categories between the category
$\CJ\mathbf{mod}_{\oH(\mathrm{w}),\psi_\mathrm{z}}$ and the category
of Fr\'{e}chet space.
\end{prp}

In view of the following lemma, Proposition \ref{factor1} is a
combination of \cite[Theorem 3.3]{du87} and \cite[Theorem
3.4]{du89}.

\begin{lemp}\label{kernel}
Denote by $I_{\psi_\mathrm{z}}$ the annihilator of
$\omega_{\psi_\mathrm{z}}$ in $\RD^{\varsigma}(\oH(\rw))$. Then
$I_{\psi_\mathrm{z}}$ annihilates every representation in
$\CJ\mathbf{mod}_{\oH(\mathrm{w}),\psi_\mathrm{z}}$. Consequently,
every representation in
$\CJ\mathbf{mod}_{\oH(\mathrm{w}),\psi_\mathrm{z}}$ is a
$\RD^{\varsigma}(\oH(\rw))/I_{\psi_\mathrm{z}}$-module.
\end{lemp}

\begin{proof}
Fix a Haar measure $dw$ on $\mathrm{w}$ and a Haar measure $dz$ on
$\mathrm{z}$. Then $dh:=dw\otimes dz$ is a Haar measure on
$\oH(\mathrm{w})$. Define a continuous surjective map
\begin{equation}\label{desc}
  \begin{array}{rcl}
    \RD^{\varsigma}(\oH(\mathrm{w}))&\rightarrow& \con^{\,\varsigma}(\mathrm{w}),\\
     \lambda=f\,dh&\mapsto &( w\mapsto \int_{\mathrm{z}}
        f(zw)\psi(z)\,dz).
  \end{array}
\end{equation}
Denote by $I'_{\psi_\mathrm{z}}$ the kernel of this map. It is
checked to be a closed ideal of $\RD^{\varsigma}(\oH(w))$. We view
$\con^{\,\varsigma}(\mathrm{w})$ as an associative algebra so that
the map \eqref{desc} is an algebra homomorphism. It is easy to see
that all representations in
$\CJ\mathbf{mod}_{\oH(\mathrm{w}),\psi_\mathrm{z}}$ are annihilated
by $I'_{\psi_\mathrm{z}}$. Therefore they are all
$\con^{\,\varsigma}(\mathrm{w})$-modules.

On the other hand, a classical result of I. Segal says that the
action of $\con^{\,\varsigma}(\mathrm{w})$ on $\omega_{\psi_\rz}$ is
faithful (c.f. \cite[page 826]{Ho80}). Therefore
$I'_{\psi_\mathrm{z}}=I_{\psi_\mathrm{z}}$. This proves the lemma.

\end{proof}

\subsection{Unitary representations}
As usual, for every complex vector space $E$, write $\bar E$ for its
complex conjugation. This equals to $E$ as a real vector space, and
its complex scalar multiplication is obtained by composing the
complex conjugation with the scalar multiplication of $E$.

Denote by $\la\,,\,\ra_{\psi_{\mathrm{z}}}$ the inner product on the
Hilbert space $\omega_{\psi_{\mathrm z}}^\rh$. It  restricts to a
nonzero $\oH(\mathrm{w})$-invariant continuous bilinear form
\[
   \la\,,\,\ra_{\psi_{\mathrm{z}}}:\omega_{\psi_\mathrm{z}}\times \bar{\omega}_{\psi_\mathrm{z}}
   \rightarrow \C.
\]
It is well known that such a form is unique up to scalar (see
\cite[Proposition 4.12]{du91} for example). This implies the following
\begin{lemp}\label{factorf}
For all Fr\'{e}chet spaces $E$ and $F$, the map
\[
   \begin{array}{rcl}
   \oB(E,F)&\rightarrow &\oB_{\oH(\mathrm{w})}(\omega_{\psi_\mathrm{z}}\widehat{\otimes}
  E, \bar \omega_{\psi_\mathrm{z}}\widehat{\otimes}F),\\
   b&\mapsto& \la\,,\,\ra_{\psi_{\mathrm{z}}}\otimes b
   \end{array}
\]
is a linear isomorphism, where ``$\oB$" stands for the space of
continuous bilinear forms, and ``$\oB_{\oH(\mathrm{w})}$" stands for
the space of $\oH(\mathrm{w})$-invariant continuous bilinear forms.
\end{lemp}

Denote by $\CH\mathbf{mod}_{\oH(\mathrm{w}),\psi_\mathrm{z}}$ the
category of unitary $\psi_\mathrm{z}$-representations of $\oH(\rw)$. For every Hilbert space $E$ with inner product
$\la\,,\,\ra_E$, $\omega_{\psi_{\mathrm z}}^\rh\widehat{\otimes}_\rh
E$ is a unitary representation in
$\CH\mathbf{mod}_{\oH(\mathrm{w}),\psi_\mathrm{z}}$. (The inner
product is $\la\,,\,\ra_{\psi_{\mathrm{z}}}\otimes \la\,,\,\ra_E$.)
Conversely, we have
\begin{lemp}\label{factorh}
Every representation $V$ in
$\CH\mathbf{mod}_{\oH(\mathrm{w}),\psi_\mathrm{z}}$ is unitarily
isomorphic to $\omega_{\psi_{\mathrm z}}^\rh\widehat{\otimes}_\rh E$
for some Hilbert space $E$.
\end{lemp}
\begin{proof}
This is well known. We provide a proof for completeness. By
Proposition \ref{factor1}, the smoothing $V_\infty$ has the form
$\omega_{\psi_{\mathrm z}}\widehat{\otimes} E_0$, where $E_0$ is a
certain  Fr\'{e}chet space. By Lemma \ref{factorf}, the restriction
of the inner product on $V$ to $\omega_{\psi_{\mathrm
z}}\widehat{\otimes} E_0$ has the form
$\la\,,\,\ra_{\psi_{\mathrm{z}}}\otimes b$, where $b$ is an inner
product on $E_0$. Denote by $E$ the completion of $E_0$ with respect
to $b$. This is a Hilbert space and we have that
$V=\omega_{\psi_{\mathrm z}}^\rh\widehat{\otimes}_\rh E$.

\end{proof}

For every unitary representation $V$ in
$\CH\mathbf{mod}_{\oH(\mathrm{w}),\psi_\mathrm{z}}$, with inner
product $\la\,,\,\ra_V$, define an inner product on $\Hom_{\oH(\rw)}
(\omega_{\psi}^{\th},V)$ by
\begin{equation}\label{inner}
  \la \phi, \phi'\ra:={\phi'}^*\circ\phi\in \Hom_{\oH(\rw)}
  (\omega_{\psi_{\mathrm z}}^\rh,\omega_{\psi_{\mathrm z}}^\rh)=\C,
\end{equation}
where ${\phi'}^*\in\Hom_{\oH(\rw)} (V,\omega_{\psi_{\mathrm z}}^\rh)$ is
determined by the formula
\[
  \la \phi'(u), v\ra_V=\la u, \phi'^*(v)\ra_{\psi_\rz}, \quad u\in \omega_{\psi_{\mathrm z}}^\rh,\, v\in V.
\]

The following analog of Proposition \ref{factor1} is an easy
consequence of Lemma \ref{factorh}.
\begin{prpp}\label{factorh2}
For every representation $V$ in
$\CH\mathbf{mod}_{\oH(\rw),\psi_\rz}$,
$\Hom_{\oH(\rw)}(\omega_{\psi_\rz}^\rh, V)$ is a Hilbert space under
the inner product \eqref{inner}. For every Hilbert space $E$,
$\omega_{\psi_\mathrm{z}}^\rh\widehat{\otimes}_\rh E$ is a
representation in $\CH\mathbf{mod}_{\oH(\rw),\psi_\rz}$. Further
more, one has inner product preserving identifications
\[
  \omega_{\psi_\mathrm{z}}^\rh\widehat{\otimes}_\rh \, \Hom_{\oH(\mathrm{w})}(\omega_{\psi_\mathrm{z}}^\rh,
  V)=V,
\]
and
\[
   \Hom_{\oH(\mathrm{w})}(\omega_{\psi_\mathrm{z}}^\rh,\omega_{\psi_\mathrm{z}}^\rh\widehat{\otimes}_\rh
   E)=E.
\]
Consequently, the functors
\[
\Hom_{\oH(W)}(\omega_{\psi_\mathrm{z}}^\rh,\,\cdot)\,
\quad\textrm{and}\quad
\omega_{\psi_\mathrm{z}}^\rh\widehat{\otimes}_\rh \,\cdot
 \]
are mutually inverse equivalences of categories between the category
$\CH\mathbf{mod}_{\oH(\mathrm{w}),\psi_\mathrm{z}}$ and the category
of Hilbert spaces.
\end{prpp}

\section{Representations of Jacobi groups}\label{jacobi}

Let $G$ be a real Jacobi group, and we return to the notation of the Introduction. This section is devoted to a proof of Theorem \ref{thm0}. Without lose of generality, we assume in this section that $\psi$ has a discrete kernel. Then $Z$ is either trivial or one dimensional. If it is trivial, then $H$ is trivial and Theorem \ref{thm0} is also trivial. So further assume in this section that $Z$ is one dimensional. Then $H$ is a Heisenberg group as in last section.

\subsection{Splitting Jacobi groups}

Fix a subspace $\rw$ of the Lie algebra $\h$ which is
complementary to $\z$. Such a space is unique up to conjugations by
$H$. It determines a Nash splitting $i_L: L\rightarrow G$ of the quotient map $G\rightarrow L$ so that $i_L(L)$ stabilizes $\rw$ under the adjoint action (cf. \cite[Theorem 7.1]{Mo}). Identify $L$ with $i_L(L)$, then we have
\[
  G=L\ltimes H.
\]

As in the Introduction, $\rw\cong\h/\z$ is a symplectic space under
the Lie bracket
\[
  [\,,\,]: \mathrm{w}\times \mathrm{w}\rightarrow \z.
\]
We let the symplectic group $\Sp(\mathrm{w})$ acts on $H$ as group
automorphisms so that it point-wise fixes $Z$ and induces the natural action on
$\mathrm{w}$. Let $\widetilde \Sp(\rw)$ acts on $H$ through the covering map $\widetilde \Sp(\rw)\rightarrow \Sp(\rw)$. Then the following semidirect product is a real Jacobi group:
\[
  \widetilde \oJ(\rw):=\widetilde{\Sp}(\mathrm{w})\ltimes H.
\]
The adjoint action induces a homomorphism
$L\rightarrow \Sp(\rw)$. Recall from the Introduction that
\[
  \widetilde{L}:=L\times_{\Sp(\mathrm{w})}\widetilde{\Sp}(\mathrm{w})
\]
is a double cover of $L$. The double cover
\[
  \widetilde G:=G\times_{L}\widetilde L =\widetilde L\ltimes H
\]
of $G$ is obviously mapped to both $\widetilde \oJ(\rw)$ and
$\widetilde L$. In this way, we view it as a subgroup of the product
$\widetilde \oJ(\rw)\times \widetilde L$.

\subsection{The oscillator representation}
To distinguish representations of different groups, we write $\pi|_S$ to emphasize that $\pi$ is viewed as a representation of a group $S$. Denote by $\omega_\psi^\rh|_{\widetilde \oJ(\rw)}$ the unitary  oscillator representation of $\widetilde \oJ(\rw)$ corresponding to $\psi$. Up
to isomorphism, this is the only genuine unitary $\psi$-representation of it which remains irreducible when restricted to $H$. Without lose of generality, assume that the representation $\omega_\psi^\rh$ of $\widetilde G$ in the introduction coincides with the pull back of $\omega_\psi^\rh|_{\widetilde \oJ(\rw)}$ to $\widetilde G$.

The following lemma is well know (see \cite[Section 2]{Ad}, for
example).
\begin{lem}\label{dosc}
The universal enveloping algebras $\oU(\s\p(\rw)_\C\ltimes \h_\C)$
and $\oU(\h_\C)$ produce the same subalgebra of
$\End((\omega_\psi^\rh|_{\widetilde \oJ(\rw)})_\infty)$.
\end{lem}

Lemma \ref{subsmooth} then implies that
\[
  \omega_\psi:=(\omega_\psi^\rh|_{\widetilde G})_\infty=(\omega_\psi^\rh|_{\widetilde \oJ(\rw)})_\infty=(\omega_\psi^\rh|_H)_\infty
\]
as Fr\'{e}chet spaces.

\subsection{Equivalences of categories}
Recall from the Introduction the categories
\[
  \CH\mathbf{mod}_{G,\psi},\,\CJ\mathbf{mod}_{G,\psi},\,\CH\mathbf{mod}_{\widetilde{L},\textrm{gen}},\,\textrm{
  and }\CJ\mathbf{mod}_{\widetilde{L},\textrm{gen}}.
\]
Given any representation $V$ in $\CH\mathbf{mod}_{G,\psi}$, recall
from the last section that  $\Hom_{H}(\omega_{\psi}^\rh,V)$ is a
Hilbert space. Let $\widetilde G$ act on it by
\[
  (\tilde g.\phi)(x):=g(\phi(\tilde g^{-1}x)), \quad \tilde g\in
  \widetilde G,\, \phi\in \Hom_{H}(\omega_{\psi}^\rh,V), \,x\in
  \omega_{\psi}^\rh,
\]
where $g$ is the image of $\tilde g$ under the quotient map
$\widetilde{G}\rightarrow G$. This action is checked to be
continuous (use Proposition \ref{factorh2}) and descends to a genuine representation of $\widetilde L$. Therefore we get a functor
\[
  \Hom_{H}(\omega_{\psi}^\rh,\cdot): \CH\mathbf{mod}_{G,\psi}\rightarrow
\CH\mathbf{mod}_{\widetilde{L},\textrm{gen}}.
 \]

On the other hand, given any representation $E$ in
$\CH\mathbf{mod}_{\widetilde{L},\textrm{gen}}$, the tensor product
$\omega_{\psi}^\rh\widehat{\otimes}_\rh E$ is a Hilbert space
representation of $\widetilde \oJ(\rw)\times \widetilde L$. Its
restriction to $\widetilde G$ descends to a representation of $G$.
In this way, we get a functor
\[
  \omega_{\psi}^\rh\widehat{\otimes}_\rh\,\cdot: \CH\mathbf{mod}_{\widetilde{L},\textrm{gen}}\rightarrow
\CH\mathbf{mod}_{G,\psi}.
\]

Similarly, we have functors
\begin{equation}\label{equi1}
  \Hom_{H}(\omega_{\psi},\cdot): \CJ\mathbf{mod}_{G,\psi}\rightarrow
\CJ\mathbf{mod}_{\widetilde{L},\textrm{gen}}
 \end{equation}
and
\begin{equation}\label{equi2}
  \omega_{\psi}\widehat{\otimes}\,\cdot: \CJ\mathbf{mod}_{\widetilde{L},\textrm{gen}}\rightarrow
\CJ\mathbf{mod}_{G,\psi}.
\end{equation}

\begin{prpl}\label{factor2}
The functors
\[
 \Hom_{H}(\omega_{\psi}^\rh,\,\cdot)\quad \textrm{and}\quad  \omega_{\psi}^\rh\widehat{\otimes}_\rh\,\cdot  \quad
 \]
are mutually inverse equivalences of categories between
$\CH\mathbf{mod}_{G,\psi}$ and
$\CH\mathbf{mod}_{\widetilde{L},\textrm{gen}}$. Similarly, the
functors
\[
 \Hom_{H}(\omega_{\psi},\,\cdot)\quad \textrm{and}\quad  \omega_{\psi}\widehat{\otimes}\,\cdot  \quad
 \]
are mutually inverse equivalences of categories between
$\CJ\mathbf{mod}_{G,\psi}$ and
$\CJ\mathbf{mod}_{\widetilde{L},\textrm{gen}}$.
\end{prpl}
\begin{proof}
In view of Proposition \ref{factorh2}, to prove the first assertion,
it suffices to show that the identification
\[
   \Hom_{\oH(W)}(\omega_{\psi}^\rh,\omega_{\psi}^\rh\widehat{\otimes}_\rh
   E)=E
\]
respects the $\widetilde L$-actions, and the identification
\[
  \omega_{\psi}^\rh\widehat{\otimes}_\rh \, \Hom_{\oH(W)}(\omega_{\psi}^\rh,
  V)=V
\]
respects the $G$-actions. This is routine to check and we will not
go to the details. The second assertion is proved similarly.
\end{proof}

\begin{corl}\label{modg}
Every representation in $\CH\mathbf{mod}_{G,\psi}$ is of moderate
growth.
\end{corl}
\begin{proof}
Since every representation in $\CH\mathbf{mod}_{G,\psi}$ has the
form $\omega_{\psi}^\rh\widehat{\otimes}_\rh
   E$ ($E$ is a representation in $\CH\mathbf{mod}_{\widetilde{L},\textrm{gen}}$), this follows from the well know fact that every Banach space
   representation of a reductive Nash group is of moderate growth (cf. \cite[Lemma 11.5.1]{Wa}).
\end{proof}

\subsection{Smoothing}

Recall that $\widetilde G$ is a Lie subgroup of $\widetilde
\oJ(\rw)\times \widetilde L$, and therefore $\g_\C$ is a Lie
subalgebra of $(\s\p(\rw)_\C\ltimes \h_\C)\times \l_\C$.

Let $E$ be a representation in
$\CH\mathbf{mod}_{\widetilde{L},\textrm{gen}}$. Write
\[
 V^\circ:=\omega_{\psi}^\rh|_{\widetilde \oJ(\rw)}\widehat{\otimes}_\rh E,
\]
viewed as a representation of $\widetilde \oJ(\rw)\times \widetilde
L$. The descent to $G$ of its restriction to $\widetilde G$ is
denoted by $V$. Lemma \ref{hs} and Lemma \ref{tensor0} implies that
\[
  V^\circ_\infty=(\omega_{\psi}^\rh|_{\widetilde \oJ(\rw)})_\infty\widehat{\otimes}_\rh E_\infty=\omega_{\psi}|_{\widetilde \oJ(\rw)}\widehat{\otimes}_\rh E_\infty=\omega_{\psi}|_{\widetilde \oJ(\rw)}\widehat{\otimes}
  E_\infty.
\]

\begin{leml}\label{subalge}
The universal enveloping algebras $\oU((\s\p(\rw)_\C\ltimes
\h_\C)\times \l_\C)$ and $\oU(\g_\C)$ produce the same subalgebra of
$\End(V^\circ_\infty)$.

\end{leml}
\begin{proof}
Write
\[
  \rho:\oU((\s\p(\rw)_\C\ltimes \h_\C)\times \l_\C)\rightarrow \End(V^\circ_\infty)
\]
for the Lie algebra action. Lemma \ref{dosc} implies that
\[
 \rho(\s\p(\rw)_\C)\subset \rho(\oU(\h_\C))\subset  \rho(\oU(\g_\C)).
\]
The lemma then follows by noting that
\[
  (\s\p(\rw)_\C\ltimes \h_\C)\times \l_\C=\s\p(\rw)_\C+\g_\C.
\]
\end{proof}

Finally, by Lemma \ref{subsmooth}, we conclude that
\[
  V_\infty=V^\circ_\infty=\omega_{\psi}\widehat{\otimes}
  E_\infty.
\]
Together with Proposition \ref{factor2}, this finishes the proof of
Theorem \ref{thm0}.

\section{Casselman-Wallach $\psi$-representations and generalized matrix coefficients}\label{matrix}

We continue to use the notation of the Introduction and assume that $G$ is a real Jacobi group.

\subsection{Casselman-Wallach $\psi$-representations}

Recall from the Introduction that $\CF\CH_{G,\psi}$ is the full subcategory of $\CJ\mathbf{mod}_{G,\psi}$ corresponding (under the functors \eqref{fu3} and \eqref{fu4})
to the subcategory $\CF\CH_{\widetilde{L},\textrm{gen}}$ of
$\CJ\mathbf{mod}_{\widetilde{L},\textrm{gen}}$. Objects in $\CF\CH_{G,\psi}$ are called Casselman-Wallach $\psi$-representations of $G$.

The following lemma is well known and is an easy consequenc of Casselman-Wallach's Theorem
(\cite[Corollary 11.6.8]{Wa}).

\begin{lem}\label{em1}
Every injective homomorphism $\phi: E\rightarrow F$ in the category
$\CF\CH_{\widetilde{L},\textrm{gen}}$ is a topological embedding
with closed image.
\end{lem}

We use the above lemma to show the following
\begin{lem}\label{autot}
Let $U$ be a smooth Fr\'{e}chet representation of $G$ of moderate
growth, and let $V$ be a Casselman-Wallach $\psi$-representation of $G$.  Then
every $G$-intertwining continuous linear map $\phi$ from $U$ to $V$
is a topological homomorphism with closed image.
\end{lem}
\begin{proof}
The map $\phi$ descends to an injective homomorphism
\[
  \phi': U/\Ker(\phi)\rightarrow V
\]
in the category $\CJ\mathbf{mod}_{G,\psi}$. By Proposition
\ref{factor2}, this corresponds to an injective homomorphism
\[
  \phi'': E\rightarrow F
\]
in the category $\CJ\mathbf{mod}_{\widetilde{L},\textrm{gen}}$.
Since $F$ is in $\CF\CH_{\widetilde{L},\textrm{gen}}$, by applying
Harish-Chandra's functor to $\phi''$, we see that $E$ is also in
$\CF\CH_{\widetilde{L},\textrm{gen}}$. Then Lemma \ref{em1} says
that $\phi''$ is a topological embedding with closed image. This
implies that so it $\phi'$, since $\omega_\psi\widehat{\otimes}
\,\cdot\,$ is a topologically exact functor on the category of
Fr\'{e}chet spaces (cf. \cite[Theorem 5.24]{Ta}).
\end{proof}

Now let $U=\omega_{\psi}\widehat \otimes E$ be a representation in
$\CF\CH_{G,\psi}$, and let $V=\bar \omega_{\psi}\widehat \otimes F$
be a representation in $\CF\CH_{G,\bar \psi}$. Here both $E$ and $F$
are representations in $\CF\CH_{\widetilde{L},\textrm{gen}}$. We say
that they are contragredient to each other if there is given a
non-degenerate $G$-invariant continuous bilinear form
\[
   \la\,,\,\ra:U\times V\rightarrow \C.
\]

\begin{prpl}\label{dual}
Every  representation in $\CF\CH_{G,\psi}$ has a unique
contragredient representation in  $\CF\CH_{G,\bar \psi}$.
\end{prpl}
\begin{proof}
Lemma \ref{factorf} implies that
\[
  \oB_G(U,V)=\oB_{\widetilde L}(E,F).
\]
Therefore, in view of Proposition \ref{factor2}, the Proposition
follows from the corresponding result in the category
$\CF\CH_{\widetilde{L},\textrm{gen}}$. But the later is a
consequence of Casselman-Wallach's Theorem.
\end{proof}

\subsection{Generalized matrix coefficients}
Let $U$, $V$, $E$, $F$ be as before. Assume that $U$ and $V$, and
hence $E$ and $F$, are contragredient to each other. It is well
known that there is a representation $E^\rh$ in
$\CH\mathbf{mod}_{\widetilde{L},\textrm{gen}}$ so that its smoothing
coincides $E$ (cf. \cite[Section 3.1]{BK}). Denote by $F^\rh$ the
space of continuous linear functionals on $E^\rh$. As usual, it is a
representation in $\CH\mathbf{mod}_{\widetilde{L},\textrm{gen}}$
which is contragredient to $E^\rh$. Put
\[
  U^\rh:=\omega_{\psi}^\rh\widehat \otimes_\rh E^\rh\quad
  \textrm{and}\quad  V^\rh:=\bar \omega_{\psi}^\rh\widehat
  \otimes_\rh
  F^\rh.
\]
View them as representations in $\CH\mathbf{mod}_{G,\psi}$ and
$\CH\mathbf{mod}_{G,\bar \psi}$, respectively. Theorem \ref{thm0}
implies that
\[
  U^\rh_\infty=U\quad
  \textrm{and}\quad  V^\rh_\infty=V.
  \]
Denote by $U^{-\infty}$ the strong dual of $V$, and by $V^{-\infty}$
the strong dual of $U$. They are both representations of $G$ (due to
the fact that both $U$ and $V$ are Nuclear Fr\'{e}chet, and hence
reflexive spaces). Furthermore, $U^{-\infty}$ contains $U$ as a
dense subspace, and $V^{-\infty}$ contains $V$ as a dense subspace.

Note that $U^\rh$ and $V^\rh$ are strongly dual to each
other. By the argument of Section \ref{smatrix}, we get a separately
continuous bilinear map
\begin{equation}\label{sepm}
  U^{-\infty}\times V^{-\infty}\rightarrow \con^{-\xi}(G)
\end{equation}
which extends the usual matrix coefficient map. Note that all the
three spaces in \eqref{sepm} are strong duals of  reflexive
Fr\'{e}chet spaces. Therefore \cite[Theorem 41.1]{Tr} implies that
the map \eqref{sepm} is automatically continuous. Then \eqref{sepm}
induces a continuous linear map
\begin{equation}\label{mapc2}
   c \,:\, U^{-\infty}\widehat \otimes V^{-\infty}\rightarrow
   \con^{-\xi}(G).
\end{equation}
By \cite[Proposition 50.7]{Tr}, $U^{-\infty}\widehat \otimes
V^{-\infty}$ equals to the strong dual of $V\widehat \otimes U$.
Therefore \eqref{mapc2} is the transpose of a $G\times
G$-intertwining continuous linear map
\begin{equation}\label{mapc3}
 \RD^{\,\varsigma}(G)\rightarrow V\widehat \otimes U.
\end{equation}
Apply Lemma \ref{autot} to the group $G\times G$, we see that the
map \eqref{mapc3} is a topological homomorphism with closed image.
In view of the following lemma, we conclude that \eqref{mapc2} is
also a topological homomorphism with closed image. This proves
Corollary \ref{thm1}.

\begin{lem}(See \cite[Section IV.2, Theorem 1]{Bo} and \cite[Lemma 3.8]{SZ})\label{top2}
Let $\phi:E_1\rightarrow E_2$ be a continuous linear map of nuclear
Fr\'{e}chet spaces. Denote by $E_1'$ and $E_2'$ the strong duals of
$E_1$ and $E_2$, respectively. Then $\phi$ is a topological
homomorphism if and only if its transpose $\phi^t: E_2'\rightarrow
E_1'$ is. When this is the case, both $\phi$ and $\phi^t$ have
closed images.
\end{lem}

\end{document}